\documentclass[draft]{amsart}
\usepackage{amsmath,amsthm}
\usepackage[english]{babel}
\usepackage{amsfonts}
\usepackage{latexsym}

\newtheorem{theorem}{Theorem}
\newtheorem{corollary}{Corollary}

\newcommand*{\cd}{(\cdot)}

\newcommand*{\wva}{\widetilde\varphi}

\newcommand*{\lpd}{L_p(\mathbb R^d)}
\newcommand*{\ld}{L_2(\mathbb R^d)}

\newcommand*{\wxi}{\widehat\xi}

\newcommand*{\wps}{\widetilde\psi}
\newcommand*{\wm}{\widehat m}
\newcommand*{\wx}{\widehat x}
\newcommand*{\wtt}{\widetilde t}
\newcommand*{\iRd}{\int_{\mathbb R^d}}
\newcommand*{\wt}{\widehat t}

\DeclareMathOperator*{\infp}{inf\vphantom p}
\DeclareMathOperator*{\vraisup}{vraisup}

\DeclareMathOperator*{\mes}{mes}

\begin{document}

\title[Recovery of differential operators]{Recovery of differential operators from a noisy Fourier transform}
\author{K.\,Yu.~Osipenko}
\address{Moscow State University\\
Institute of Information Transmission Problems
(Kharkevich Institute), RAS, Moscow}
\email{kosipenko@yahoo.com}
\date{08.02.2024}

\begin{abstract}
The paper concerns problems of the recovery of differential operators from a noisy Fourier transform. In particular, optimal methods are obtained for the recovery of powers of generalized Laplace operators from a noisy Fourier transform in the $L_2$-metric.
\end{abstract}

\maketitle

\section{Introduction}

Let $X$ be a linear space, $Y,Z$ be normed linear spaces. The problem of optimal recovery of the linear operator $\Lambda\colon X\to Z$ by inaccurately given values of the linear operator $I\colon X\to Y$ on the set $W\subset X$ is posed as a problem of finding the value
$$E(\Lambda,W,I,\delta)=\infp_{\varphi\colon Y\to Z}\sup_{\substack{x\in W,\ y\in Y\\\|Ix-y\|_Y
\le\delta}}\|\Lambda x-\varphi(y)\|_Z,$$
called the {\it error of optimal recovery}, and the mapping $\varphi$ on which the lower bound is attained, called the {\it optimal recovery method} (here $\delta\ge0$ is a parameter that characterizes the error of setting the values of the operator $I$). Initially, this problem was posed for the case when $\Lambda$ is a linear functional, $Y$ is a finite-dimensional space and the information is known exactly ($\delta=0$), by S.~A.~Smolyak \cite{Sm}. In fact, this statement was a generalization of A.~N.~Kolmogorov's problem about the best quadrature formula on the class of functions \cite{N1}, in which the integral and the values of the functions are replaced by arbitrary linear functionals and there is no condition for the linearity of the recovery method.
Subsequently, much research has been devoted to extensions of this problem (see \cite{MR}--\cite{Osbo}, and the references given therein).

One of the first papers in which the problem of constructing an optimal recovery method for a linear operator was considered was the paper \cite{MM}. This topic was further developed in the papers \cite{MO02}--\cite{Os5}. It turned out that in some cases it is possible to construct a whole family of optimal recovery methods for a linear operator. The study of such families began in \cite{MO3} and continued in \cite{MO4}, \cite{MO5}, \cite{Os4}, \cite{Os5}. Some general approach to constructing of family of optimal recovery methods was proposed in \cite{Os24}.

The aim of this paper is to construct families of optimal recovery methods for powers of generalized Laplace operators and the Weil derivative from a noisy Fourier transform in the $L_2$-metric.

\section{Optimal recovery methods from a noisy Fourier transform}

Let $S$ be the Schwartz space of rapidly decreasing $C^\infty$-functions on  $\mathbb R^d$, $S'$ be the corresponding space of distributions, and let
$F\colon S'\to S'$ be the Fourier transform. Set
\begin{equation*}
X_p=\left\{\,x\cd\in S':\varphi\cd Fx\cd\in\ld,\ Fx\cd\in\lpd\,\right\}.
\end{equation*}
We define the operator $D$ as follows
$$Dx\cd=F^{-1}(\varphi\cd Fx\cd)\cd.$$
Put
\begin{equation}\label{Lam}
\Lambda x\cd=F^{-1}(\psi\cd Fx\cd)\cd.
\end{equation}

Consider the problem of the optimal recovery of values of the operator $\Lambda$ on the
class
\begin{equation*}
W_p=\left\{\,x\cd\in X_p:\|Dx\cd\|_{\ld}\le1\,\right\}
\end{equation*}
from the noisy Fourier transform of the function $x\cd$. Assume that $\Lambda x\cd\in\ld$ for all $x\cd\in X_p$. As recovery methods we consider all possible mappings $m\colon\lpd\to\ld$. The error of a method $m$ is defined by
\begin{equation*}
e_p(\Lambda,D,m)=\sup_{\substack{x\cd\in W_p,\ y\cd\in\lpd\\\|Fx\cd-y\cd\|_{\lpd}\le\delta}}\|\Lambda x\cd-m(y)\cd\|_{\ld}.
\end{equation*}
The quantity
\begin{equation}\label{ED}
E_p(\Lambda,D)=\inf_{m\colon\lpd\to\ld}e_p(\Lambda,D,m)
\end{equation}
is called the error of optimal recovery, and the method on which the infimum is
attained, an optimal method.

It is easily checked that
\begin{equation}\label{eq1.1}
E_p(\Lambda,D)\ge\sup_{\substack{x\cd\in W_p\\\|Fx\cd\|_{\lpd}\le\delta}}\|\Lambda x\cd\|_{\ld}.
\end{equation}
Indeed, let $x\cd\in W_p$, $\|Fx\cd\|_{\lpd}\le\delta$, and let $m\colon\lpd\to\ld$ be an arbitrary recovery method.
Since $x\cd\in W_p$ and $-x\cd\in W_p$, we have
$$2\|\Lambda x\cd\|_{\ld}\le\|\Lambda x\cd-m(0)\cd\|_{\ld}
+\|-\Lambda x\cd-m(0)\cd\|_{\ld}\le2e_p(\Lambda,D,m).$$
It follows that, for any method~$m$,
$$e_p(\Lambda,D,m)\ge\sup_{\substack{x\cd\in W_p\\\|Fx\cd\|_{\lpd}\le\delta}}\|\Lambda x\cd\|_{\ld}.$$
Now the required inequality follows by taking the lower bound on the left over all methods.

\section{Optimal recovery methods for $\Lambda_\theta^{\eta/2}$}

Consider the polar transformation in $\mathbb R^d$
\begin{equation*}
\arraycolsep=0.08em
\begin{array}{rcl}
t_1&=&\rho\cos\omega_1,\\
t_2&=&\rho\sin\omega_1\cos\omega_2,\\
\hdotsfor{3}\\
t_{d-1}&=&\rho\sin\omega_1\sin\omega_2\ldots\sin\omega_{d-2}\cos\omega_{d-1},\\
t_d&=&\rho\sin\omega_1\sin\omega_2\ldots\sin\omega_{d-2}\sin\omega_{d-1}.
\end{array}
\end{equation*}
Set $\omega=(\omega_1,\ldots,\omega_{d-1})$. For any function $f\cd$ we put
$$\widetilde f(\omega)=|f(\cos\omega_1,\ldots,\sin\omega_1\sin\omega_2\ldots
\sin\omega_{d-2}\sin\omega_{d-1})|.$$
Note that if $|f\cd|$ is a homogenous function of degree $\kappa$, then $\widetilde f(\omega)=\rho^{-\kappa}|f(t)|$.

Let $|\psi\cd|$ be homogenous function of degree $\eta$ and $|\varphi\cd|$ be homogenous functions of degrees $\nu$, $\psi(t)\ne0$ and $\varphi_j(t)\ne0$ for almost all $t\in\mathbb R^d$. Set
\begin{gather*}
\gamma=\frac{\nu-\eta}{\nu+d(1/2-1/p)},\quad q^*=\frac1{\gamma(1/2-1/p)},\\
C_p(\nu,\eta)=\gamma^{-\frac\gamma p}
(1-\gamma)^{-\frac{1-\gamma}2}\Biggl(\frac{B\left(q^*\gamma/p+1,
q^*(1-\gamma)/2\right)}{2|\nu-\eta|}\Biggr)^{1/q^*},
\end{gather*}
where $B(\cdot,\cdot)$ is the Euler beta-function.

It follows from \cite[Theorem~6]{Os244} (see also \cite[Theorem~3]{Os5}) the following result

\begin{theorem}\label{pred1}
Let $2<p\le\infty$, $\gamma\in(0,1)$. Assume that
\begin{equation}\label{Ieq}
I=\int_{\Pi^{d-1}}\frac{\wps^{q^*}(\omega)}{\wva^{q^*(1-\gamma)}(\omega)}
J(\omega)\,d\omega<\infty,\quad\Pi^{d-1}=[0,\pi]^{d-2}\times[0,2\pi].
\end{equation}
Then
$$E_p(\Lambda,D)=\frac1{(2\pi)^{d\gamma/2}}C_p(\nu,\eta)I^{1/q^*}\delta^\gamma.$$
The method
$$\wm(y)(t)=F^{-1}\left(\left(1-\beta\frac{|\varphi(t)|^2}{|\psi(t)|^2}\right)_+
\psi(t)y(t)\right),$$
where
\begin{equation*}
\beta=\frac{1-\gamma}{(2\pi)^{d\gamma}}C_p^2(\nu,\eta)\left(\delta I^{1/2-1/p}\right)^{2\gamma},
\end{equation*}
is optimal.

Moreover, the sharp inequality
$$\|\Lambda x\cd\|_{\ld}
\le\frac{C_p(\nu,\eta)I^{1/q^*}}{(2\pi)^{d\gamma/2}}\|Fx\cd\|_{\lpd}^\gamma
\|Dx\cd\|_{\ld}^{1-\gamma}$$
holds.
\end{theorem}

Put
\begin{equation*}
\psi_\theta(\xi)=(|\xi_1|^\theta+\ldots+|\xi_d|^\theta)^{2/\theta},\quad\theta>0.
\end{equation*}
We denote by $\Lambda_\theta^{\eta/2}$ the operator $\Lambda$ which is defined by \eqref{Lam} for $\psi\cd=\psi_\theta^{\eta/2}\cd$. In particular, $\Lambda_2=-\Delta$, where $\Delta$ is the Laplace operator.

Consider problem \eqref{ED} for $\Lambda=\Lambda_\theta^{\eta/2}$ and $D=\Lambda_\mu^{\nu/2}$. Then for $I$ from \eqref{Ieq} we have
\begin{equation}\label{II}
I=\int_{\Pi^{d-1}}\frac{\left(\sum_{k=1}^d
\wtt_k^\theta(\omega)\right)^{\eta q^*/\theta}}{\left(\sum_{k=1}^d
\wtt_k^{\,\mu}(\omega)\right)^{\nu q^*(1-\gamma)/\mu}}J(\omega)\,d\omega,
\end{equation}
where
\begin{equation*}
\arraycolsep=0.08em
\begin{array}{rcl}
\wtt_1(\omega)&=&\cos\omega_1,\\
\wtt_2(\omega)&=&\sin\omega_1\cos\omega_2,\\
\hdotsfor{3}\\
\wtt_{d-1}(\omega)&=&\sin\omega_1\sin\omega_2\ldots\sin\omega_{d-2}\cos\omega_{d-1},\\
\wtt_d(\omega)&=&\sin\omega_1\sin\omega_2\ldots\sin\omega_{d-2}\sin\omega_{d-1}.
\end{array}
\end{equation*}
Note that
\begin{equation*}
\sum_{k=1}^d{\wtt_k\hspace{-3pt}}^2(\omega)=1.
\end{equation*}
If $\mu\le2$, then
\begin{equation}\label{etd1}
\sum_{k=1}^d
{\wtt_k\hspace{-3pt}}^\mu(\omega)\ge\sum_{k=1}^d
{\wtt_k\hspace{-3pt}}^2(\omega)=1.
\end{equation}
For $\mu>2$ by H\"older's inequality
\begin{equation*}
1=\sum_{k=1}^d
{\wtt_k\hspace{-3pt}}^2(\omega)\le\biggl(\sum_{k=1}^d
{\wtt_k\hspace{-3pt}}^\mu(\omega)\biggr)^{\frac2\mu}d^{1-\frac2{\mu}}.
\end{equation*}
Thus,
\begin{equation}\label{etd2}
\sum_{k=1}^d
{\wtt_k\hspace{-3pt}}^\mu(\omega)\ge d^{1-\frac\mu2}.
\end{equation}
It follows by \eqref{etd1} and \eqref{etd2} that $I<\infty$.

\begin{corollary}\label{cor}
Let $2<p\le\infty$, $\nu>\eta\ge0$, and $\theta,\mu>0$. Then
$$E_p(\Lambda_\theta^{\eta/2},\Lambda_\mu^{\nu/2})=\frac1{(2\pi)^{d\gamma/2}}C_p(\nu,\eta)I^{1/q^*}
\delta^{\gamma},$$
where $I$ is defined by \eqref{II}.
The method
$$\wm(y)(t)=F^{-1}\left(\left(1-\beta\frac{\psi_\mu^\nu(t)}{\psi_\theta^\eta(t)}\right)_+
\psi_\theta^{\eta/2}(t)y(t)\right),$$
where
\begin{equation*}
\beta=\frac{1-\gamma}{(2\pi)^{d\gamma}}C_p^2(\nu,\eta)\left(\delta I^{1/2-1/p}\right)^{2\gamma},
\end{equation*}
is optimal.

Moreover, the sharp inequality
$$\|\Lambda_\theta^{\eta/2}x\cd\|_{\ld}
\le\frac{C_p(\nu,\eta)I^{1/q^*}}{(2\pi)^{d\gamma/2}}\|Fx\cd\|_{\lpd}^\gamma
\|\Lambda_\mu^{\nu/2}x\cd\|_{\ld}^{1-\gamma}$$
holds.
\end{corollary}

Now we consider the case when $p=2$.

\begin{theorem}\label{T2}
Let $\nu>\eta>0$ and $0<\theta\le\mu$. Then
\begin{equation}\label{E}
E_2(\Lambda_\theta^{\eta/2},\Lambda_\mu^{\nu/2})= d^{\eta(1/\theta-1/\mu)}\left(\frac{\delta}{(2\pi)^{d/2}}\right)^{1-\eta/\nu},
\end{equation}
and all methods
\begin{equation}\label{mma}
\wm(y)(t)=F^{-1}\left(a(t)\psi_\theta^{\eta/2}(t)y(t)\right),
\end{equation}
where $a\cd$ are measurable functions satisfying the condition
\begin{equation}\label{aa}
\psi_\theta^\eta(\xi)\left(\frac{|1-a(\xi)|^2}{\lambda_2\psi_\mu^\nu(\xi)}+\frac{|a(\xi)|^2}
{(2\pi)^d\lambda_1}\right)\le1,
\end{equation}
in which
\begin{equation*}
\lambda_1=\frac{d^{2\eta(1/\theta-1/\mu)}}{(2\pi)^d}\left(1-\frac\eta\nu\right)
\left(\frac{(2\pi)^d}{\delta^2}
\right)^{\eta/\nu},
\quad\lambda_2=d^{2\eta(1/\theta-1/\mu)}\frac\eta\nu \left(\frac{(2\pi)^d}{\delta^2}
\right)^{\eta/\nu-1},
\end{equation*}
are optimal.

The sharp inequality
\begin{equation}\label{Sl2}
\|\Lambda_\theta^{\eta/2}x\cd\|_{\ld}
\le\frac{d^{\eta(1/\theta-1/\mu)}}{(2\pi)^{d(1-\eta/\nu)/2}}
\|Fx\cd\|_{\ld}^{1-\eta/\nu}\|\Lambda_\mu^{\nu/2}x\cd\|_{\ld}^{\eta/\nu}
\end{equation}
holds.
\end{theorem}

\begin{proof}
It follows from \eqref{eq1.1} that
\begin{equation}\label{EDD2}
E_2(\Lambda_\theta^{\eta/2},\Lambda_\mu^{\nu/2})\ge\sup_{\substack{x\cd\in W_2\\\|Fx\cd\|_{\ld}\le\delta}}\|\Lambda_\theta^{\eta/2}x\cd\|_{\ld}.
\end{equation}
Consider the extremal problem
$$\|\Lambda_\theta^{\eta/2}x\cd\|_{\ld}^2\to\max,\quad\|Fx\cd\|_{\ld}^2\le\delta^2,\quad
\|\Lambda_\mu^{\nu/2}x\cd\|_{\ld}^2\le1.$$
Given $0<\varepsilon<d^{-1/\mu}(2\pi)^{d/\nu}\delta^{-2/\nu}$, we set
\begin{equation*}
\wxi_\varepsilon=\frac1{d^{1/\mu}}\left(\frac{(2\pi)^d}{\delta^2}\right)^{\frac1{2\nu}}
(1,\ldots,1)-(\varepsilon,\ldots,\varepsilon),\quad B_\varepsilon=\{\xi\in\mathbb R^d:|\xi-\wxi_\varepsilon|<\varepsilon\,\}.
\end{equation*}
Consider a function $x_\varepsilon\cd$ such that
$$Fx_\varepsilon(\xi)=\begin{cases}\dfrac\delta{\sqrt{\mes B_\varepsilon}},&\xi\in B_\varepsilon,\\
0,&\xi\notin B_\varepsilon.\end{cases}$$
Then $\|Fx_\varepsilon\cd\|^2_{\ld}=\delta^2$ and
\begin{equation*}
\|\Lambda_\mu^{\nu/2}x\cd\|^2_{\ld}
=\frac{\delta^2}{(2\pi)^d\mes B_\varepsilon}\int_{B_\varepsilon}(|\xi_1|^\mu+\ldots+|\xi_d|^\mu)^{2\nu/\mu}\,d\xi
\le1.
\end{equation*}
By virtue of \eqref{EDD2} we have
\begin{multline*}
E_2^2(\Lambda_\theta^{\eta/2},\Lambda_\mu^{\nu/2})\ge\|\Lambda_\theta^{\eta/2}x_\varepsilon\cd\|_{L_2(\mathbb R^d)}^2\\
=\frac{\delta^2}{(2\pi)^d\mes B_\varepsilon}\int_{B_\varepsilon}\psi_\theta^\eta(\xi)\,d\xi
=\frac{\delta^2}{(2\pi)^d}\psi_\theta^\eta(\widetilde\xi_\varepsilon),\quad
\widetilde\xi_\varepsilon\in B_\varepsilon.
\end{multline*}
Letting $\varepsilon\to0$ we obtain the estimate
\begin{equation}\label{EE}
E_2^2(\Lambda_\theta^{\eta/2},\Lambda_\mu^{\nu/2})\ge d^{2\eta(1/\theta-1/\mu)}\left(\frac{\delta^2}{(2\pi)^d}\right)^{1-\eta/\nu}.
\end{equation}

We will find optimal methods among methods \eqref{mma}. Passing to the Fourier transform we have
\begin{equation*}
\|\Lambda_\theta^{\eta/2}x\cd-\wm(y)\cd\|^2_{\ld}\\
=\frac1{(2\pi)^d}\iRd\psi_\theta^\eta(\xi)\left|Fx(\xi)-a(\xi)y(\xi)\right|^2\,d\xi.
\end{equation*}
We set $z\cd=Fx\cd-y\cd$ and note that
\begin{equation*}
\iRd|z(\xi)|^2\,d\xi\le\delta^2,\quad\frac1{(2\pi)^d}\iRd\psi_\mu^\nu(\xi)|Fx(\xi)|^2\,d\xi\le1.
\end{equation*}
Then
\begin{equation*}
\|\Lambda_\theta^{\eta/2}x\cd-\wm(y)\cd\|^2_{\ld}
=\frac1{(2\pi)^d}\iRd\psi_\theta^\eta(\xi)\left|\left(1-a(\xi)\right)Fx(\xi)+
a(\xi)z(\xi)\right|^2\,d\xi.
\end{equation*}
We write the integrand as
\begin{equation*}
\psi_\theta^\eta(\xi)\left|\frac{(1-a(\xi))\sqrt{\lambda_2}\psi_\mu^{\nu/2}(\xi)
Fx(\xi)}{\sqrt{\lambda_2}\psi_\mu^{\nu/2}(\xi)}+
\frac{a(\xi)}{(2\pi)^{d/2}\sqrt{\lambda_1}}(2\pi)^{d/2}\sqrt{\lambda_1}z(\xi)\right|^2.
\end{equation*}
Applying the Cauchy-Bunyakovskii-Schwarz inequality we obtain the estimate
\begin{multline*}
\|\Lambda_\theta^{\eta/2}x\cd-\wm(y)\cd\|^2_{\ld}\\
\le\vraisup_{\xi\in\mathbb R^d}S(\xi)\frac1{(2\pi)^d}\iRd\left(\lambda_2\psi_\mu^\nu(\xi)|Fx(\xi)|^2+
(2\pi)^d\lambda_1|z(\xi)|^2\right)\,d\xi,
\end{multline*}
where
\begin{equation*}
S(\xi)=\psi_\theta^\eta(\xi)\left(\frac{|1-a(\xi)|^2}{\lambda_2\psi_\mu^\nu(\xi)}+
\frac{|a(\xi)|^2}{(2\pi)^d\lambda_1}\right).
\end{equation*}
If we assume that $S(\xi)\le1$ for almost all $\xi$, then taking into account \eqref{EE}, we get
\begin{multline}\label{Ub}
e^2_2(\Lambda_\theta^{\eta/2},\Lambda_\mu^{\nu/2},\wm)\\
\le\frac1{(2\pi)^d}\iRd\left(\lambda_2\psi_\mu^\nu(\xi)|Fx(\xi)|^2+
(2\pi)^d\lambda_1|z(\xi)|^2\right)\,d\xi
\le\lambda_2+\lambda_1\delta^2\\=d^{2\eta(1/\theta-1/\mu)}\left(\frac{\delta^2}
{(2\pi)^d}\right)^{1-\eta/\nu}\le E_2^2(\Lambda_\theta^{\eta/2},\Lambda_\mu^{\nu/2}).
\end{multline}
This proves \eqref{E} and shows that the methods under consideration are optimal.

It remains to verify that the set of functions $a\cd$ satisfying \eqref{aa} is nonempty. Put
\begin{equation*}
a(\xi)=\frac{(2\pi)^d\lambda_1}{(2\pi)^d\lambda_1+\lambda_2\psi_\mu^\nu(\xi)}.
\end{equation*}
Then
\begin{equation*}
S(\xi)=\frac{\psi_\theta^\eta(\xi)}{(2\pi)^d\lambda_1+\lambda_2\psi_\mu^\nu(\xi)}.
\end{equation*}
Since $\theta\le\mu$ by H\"older's inequality
\begin{equation*}
\sum_{j=1}^d|\xi_j|^\theta\le\biggl(\sum_{j=1}^d|\xi_j|^\mu\biggr)^{\theta/\mu}
d^{1-\theta/\mu}.
\end{equation*}
Putting $\rho=(|\xi_1|^\theta+\ldots+|\xi_d|^\theta)^{1/\theta}$, we obtain
\begin{equation*}
\sum_{j=1}^d|\xi_j|^\mu\ge\rho^\mu d^{1-\mu/\theta}.
\end{equation*}
Thus,
\begin{equation*}
S(\xi)\le\frac{\rho^{2\eta}}{(2\pi)^d\lambda_1+\lambda_2\rho^{2\nu}d^{2\nu(1/\mu-1/\theta)}}.
\end{equation*}
It is easily checked that the function $f(\rho)=(2\pi)^d\lambda_1+\lambda_2\rho^{2\nu}d^{2\nu(1/\mu-1/\theta)}-\rho^{2\eta}$ reaches a minimum on  $[0,+\infty)$ at
\begin{equation*}
\rho_0=d^{1/\theta-1/\mu}\left(\frac{(2\pi)^d}{\delta^2}
\right)^{1/(2\nu)}.
\end{equation*}
Moreover, $f(\rho_0)=0$. Consequently, $f(\rho)\ge0$ for all $\rho\ge0$. Hence $S(\xi)\le1$ for all $\xi$.

Let $x\cd\in X_2$ for $\varphi\cd=\psi_\mu^{\nu/2}\cd$. Put $A=\|\Lambda_\mu^{\nu/2}x\cd\|_{\ld}+\varepsilon$, $\varepsilon>0$. Consider $\wx\cd=x\cd/A$. Put $\delta=\|F\wx\cd\|_{\ld}$. It follows from \eqref{Ub} that
\begin{equation}\label{EEpqr}
\sup_{\substack{x\cd\in W_2\\\|Fx\cd\|_{\ld}\le\delta}}\|\Lambda_\theta^{\eta/2}x\cd\|_{\ld}=
E_2(\Lambda_\theta^{\eta/2},\Lambda_\mu^{\nu/2})=d^{\eta(1/\theta-1/\mu)}\left(\frac{\delta}
{(2\pi)^{d/2}}\right)^{1-\eta/\nu}.
\end{equation}
Thus,
$$\|\Lambda_\theta^{\eta/2}\wx\cd\|_{\ld}\le d^{\eta(1/\theta-1/\mu)}\left(\frac{\delta}
{(2\pi)^{d/2}}\right)^{1-\eta/\nu}.$$
Consequently,
$$\|\Lambda_\theta^{\eta/2}x\cd\|_{\ld}\le\frac{d^{\eta(1/\theta-1/\mu)}}{(2\pi)^{d(1-\eta/\nu)/2}}
\|Fx\cd\|_{\ld}^{1-\eta/\nu}\left(\|\Lambda_\mu^{\nu/2}x\cd\|_{\ld}+
\varepsilon\right)^{\eta/\nu}.$$
Letting $\varepsilon\to0$ we obtain \eqref{Sl2}.

If there exists a
$$C<\frac{d^{\eta(1/\theta-1/\mu)}}{(2\pi)^{d(1-\eta/\nu)/2}}$$
for which
$$\|\Lambda_\theta^{\eta/2}x\cd\|_{\ld}
\le C\|Fx\cd\|_{\ld}^{1-\eta/\nu}\|\Lambda_\mu^{\nu/2}x\cd\|_{\ld}^{\eta/\nu},$$
then
$$\sup_{\substack{x\cd\in W_2\\\|Fx\cd\|_{\ld}\le\delta}}\|\Lambda_\theta^{\eta/2}x\cd\|_{\ld}\le C\delta^{1-\eta/\nu}<\frac{d^{\eta(1/\theta-1/\mu)}}{(2\pi)^{d(1-\eta/\nu)/2}}
\delta^{1-\eta/\nu}.$$
This contradicts with \eqref{EEpqr}.
\end{proof}

Let $\alpha=(\alpha_1,\ldots,\alpha_d)\in\mathbb R_+^d$. We define the operator $D^\alpha$ (the derivative of order $\alpha$) by
\begin{equation*}
D^\alpha x\cd=F^{-1}((i\xi)^\alpha Fx(\xi))\cd,
\end{equation*}
where $(i\xi)^\alpha=(i\xi_1)^{\alpha_1}\ldots(i\xi_d)^{\alpha_d}$. For $\mu=2\nu$ we have
\begin{multline*}
\|\Lambda_{2\nu}^{\nu/2}x\cd\|_{\ld}^2=\frac1{(2\pi)^d}\iRd\left(|\xi_1|^{2\nu}+\ldots+
|\xi_d|^{2\nu}\right)|Fx(\xi)|^2\,d\xi\\
=\sum_{j=1}^d\|D^{\nu e_j}x\cd\|_{\ld}^2,
\end{multline*}
where $e_j$, $j=1\ldots,d$, is a standard basis in $\mathbb R^d$.

From Theorem~\ref{T2} we obtain the following result.

\begin{corollary}
Let $\nu>\eta>0$ and $0<\theta\le2\nu$. Then
$$E_2(\Lambda_\theta^{\eta/2},\Lambda_{2\nu}^{\nu/2})= d^{\eta(1/\theta-1/(2\nu))}\left(\frac{\delta}{(2\pi)^{d/2}}\right)^{1-\eta/\nu},$$
and all methods
$$\wm(y)(t)=F^{-1}\left(a(t)\psi_\theta^{\eta/2}(t)y(t)\right),$$
where $a\cd$ are measurable functions satisfying the condition
$$\psi_\theta^\eta(\xi)\left(\frac{|1-a(\xi)|^2}{\lambda_2\psi_{2\nu}^\nu(\xi)}+\frac{|a(\xi)|^2}
{(2\pi)^d\lambda_1}\right)\le1,$$
in which
\begin{align*}
\lambda_1&=\frac{d^{2\eta(1/\theta-1/(2\nu))}}{(2\pi)^d}\left(1-\frac\eta\nu\right)
\left(\frac{(2\pi)^d}{\delta^2}
\right)^{\eta/\nu},\\
\quad\lambda_2&=d^{2\eta(1/\theta-1/(2\nu))}\frac\eta\nu\left(\frac{(2\pi)^d}{\delta^2}
\right)^{\eta/\nu-1},
\end{align*}
are optimal.

The sharp inequality
\begin{equation}\label{Sl2a}
\|\Lambda_\theta^{\eta/2}x\cd\|_{\ld}
\le\frac{d^{\eta(1/\theta-1/(2\nu))}}{(2\pi)^{d(1-\eta/\nu)/2}}
\|Fx\cd\|_{\ld}^{1-\eta/\nu}\biggl(\sum_{j=1}^d\|D^{\nu e_j}x\cd\|_{\ld}^2\biggr)^{\eta/(2\nu)}
\end{equation}
holds.
\end{corollary} 

For integer $\nu$ inequality \eqref{Sl2a} can be rewritten in the form
$$\|\Lambda_\theta^{\eta/2}x\cd\|_{\ld}
\le\frac{d^{\eta(1/\theta-1/(2\nu))}}{(2\pi)^{d(1-\eta/\nu)/2}}
\|Fx\cd\|_{\ld}^{1-\eta/\nu}\biggl(\sum_{j=1}^d\biggl\|\frac{\partial^\nu x}{\partial t_j^\nu}\cd\biggr\|_{\ld}^2\biggr)^{\eta/(2\nu)}.$$

\section{Optimal recovery methods for $D^\alpha$}

Now we consider problem \eqref{ED} for $\Lambda=D^\alpha$ and $D=\Lambda_\mu^{\nu/2}$. Then for $I$ from \eqref{Ieq} we have
\begin{equation}\label{IIa}
I=\int_{\Pi^{d-1}}\frac{\left(\wtt_1^{\alpha_1}(\omega)\ldots\wtt_d^{\alpha_d}(\omega)
\right)^{q_1}}{\left(\sum_{k=1}^d
\wtt_k^{\,\mu}(\omega)\right)^{\nu q_1(1-\gamma_1)/\mu}}J(\omega)\,d\omega,
\end{equation}
where
$$\gamma_1=\frac{\nu-|\alpha|}{\nu+d(1/2-1/p)},\quad q_1=\frac1{\gamma_1(1/2-1/p)},\quad|\alpha|=\alpha_1+\ldots+\alpha_d.$$
It follows by \eqref{etd1} and \eqref{etd2} that $I<\infty$.

\begin{corollary}\label{cor1}
Let $2<p\le\infty$, $\nu>|\alpha|\ge0$, and $\mu>0$. Then
$$E_p(D^\alpha,\Lambda_\mu^{\nu/2})=\frac1{(2\pi)^{d\gamma_1/2}}C_p(\nu,|\alpha|)I^{1/q_1}
\delta^{\gamma_1},$$
where $I$ is defined by \eqref{IIa}.
The method
$$\wm(y)(t)=F^{-1}\left(\left(1-\beta\frac{\psi_\mu^\nu(t)}{t_1^{2\alpha_1}\ldots t_d^{2\alpha_d}}\right)_+
(it)^\alpha y(t)\right),$$
where
\begin{equation*}
\beta=\frac{1-\gamma_1}{(2\pi)^{d\gamma_1}}C_p^2(\nu,|\alpha|)\left(\delta I^{1/2-1/p}\right)^{2\gamma_1},
\end{equation*}
is optimal.

Moreover, the sharp inequality
$$\|D^\alpha x\cd\|_{\ld}
\le\frac{C_p(\nu,|\alpha|)I^{1/q_1}}{(2\pi)^{d\gamma_1/2}}\|Fx\cd\|_{\lpd}^{\gamma_1}
\|\Lambda_\mu^{\nu/2}x\cd\|_{\ld}^{1-\gamma_1}$$
holds.
\end{corollary}

Consider the case when $p=2$.

\begin{theorem}\label{T3}
Let $2\nu\ge\mu>0$ and $0<|\alpha|<\nu$. Then
\begin{equation}\label{Ea}
E_2(D^\alpha,\Lambda_\mu^{\nu/2})= \left(\frac{\delta}{(2\pi)^{d/2}}\right)^{1-|\alpha|/\nu}|\alpha|^{-|\alpha|/\mu}
\prod_{\substack{j=1\\\alpha_j\ne0}}^d\alpha_j^{\alpha_j/\mu},
\end{equation}
and all methods
\begin{equation}\label{mmaa}
\wm(y)(t)=F^{-1}\left(a(t)(it)^\alpha y(t)\right),
\end{equation}
where $a\cd$ are measurable functions satisfying the condition
\begin{equation}\label{aaa}
\psi_\theta^\eta(\xi)\left(\frac{|1-a(\xi)|^2}{\lambda_2\psi_\mu^\nu(\xi)}+\frac{|a(\xi)|^2}
{(2\pi)^d\lambda_1}\right)\le1,
\end{equation}
in which
\begin{align*}
\lambda_1&=\frac{|\alpha|^{-2|\alpha|/\mu}}{(2\pi)^d}\left(1-\frac{|\alpha|}\nu\right)
\left(\frac{(2\pi)^d}{\delta^2}
\right)^{|\alpha|/\nu}\prod_{\substack{j=1\\\alpha_j\ne0}}^d\alpha_j^{2\alpha_j/\mu},\\
\lambda_2&=\frac{|\alpha|^{-2|\alpha|/\mu+1}}\nu\left(\frac{(2\pi)^d}{\delta^2}
\right)^{|\alpha|/\nu-1}\prod_{\substack{j=1\\\alpha_j\ne0}}^d\alpha_j^{2\alpha_j/\mu},
\end{align*}
are optimal.

The sharp inequality
\begin{equation}\label{Sl2aa}
\|D^\alpha x\cd\|_{\ld}
\le\frac{|\alpha|^{-|\alpha|/\mu}}{(2\pi)^{d(1-|\alpha|/\nu)/2}}
\prod_{\substack{j=1\\\alpha_j\ne0}}^d\alpha_j^{\alpha_j/\mu}
\|Fx\cd\|_{\ld}^{1-|\alpha|/\nu}\|\Lambda_\mu^{\nu/2}x\cd\|_{\ld}^{|\alpha|/\nu}
\end{equation}
holds.
\end{theorem}

\begin{proof}
It follows from \eqref{eq1.1} that
\begin{equation}\label{EDD2a}
E_2(D^\alpha,\Lambda_\mu^{\nu/2})\ge\sup_{\substack{x\cd\in W_2\\\|Fx\cd\|_{\ld}\le\delta}}\|D^\alpha x\cd\|_{\ld}.
\end{equation}
Consider the extremal problem
$$\|D^\alpha x\cd\|_{\ld}^2\to\max,\quad\|Fx\cd\|_{\ld}^2\le\delta^2,\quad
\|\Lambda_\mu^{\nu/2}x\cd\|_{\ld}^2\le1.$$
Given 
$$0<\varepsilon<\min\left\{|\alpha|^{-1/\mu}\left(\frac{(2\pi)^d}{\delta^2}\right)^{1/(2\nu)}
\alpha_j^{1/\mu}:
\alpha_j>0,\ j=1,\ldots,d\right\},$$
we set
\begin{multline*}
\wxi_\varepsilon=|\alpha|^{-1/\mu}\left(\frac{(2\pi)^d}{\delta^2}\right)^{1/(2\nu)}
(\alpha_1^{1/\mu},\ldots,\alpha_d^{1/\mu})-(\varepsilon_1,
\ldots,\varepsilon_d),\quad\varepsilon_j=\begin{cases}\varepsilon,&\alpha_j>0,\\
0,&\alpha_j=0,\end{cases},\\
B_\varepsilon=\{\xi\in\mathbb R^d:|\xi-\wxi_\varepsilon|<\varepsilon\,\}.
\end{multline*}
Consider a function $x_\varepsilon\cd$ such that
$$Fx_\varepsilon(\xi)=\begin{cases}\dfrac\delta{\sqrt{\mes B_\varepsilon}}\left(1+d\varepsilon^\nu\left(\dfrac{\delta^2}{(2\pi)^d}\right)^{\mu/(2\nu)}
\right)^{-1/2},&\xi\in B_\varepsilon,\\
0,&\xi\notin B_\varepsilon.\end{cases}$$
Then 
$$\|Fx_\varepsilon\cd\|^2_{\ld}=\delta^2\left(1+d\varepsilon^\nu
\left(\dfrac{\delta^2}{(2\pi)^d}\right)^{\mu/(2\nu)}\right)^{-1}\le\delta^2$$
and
\begin{multline*}
\|\Lambda_\mu^{\nu/2}x\cd\|^2_{\ld}\\
=\frac{\delta^2}{(2\pi)^d\mes B_\varepsilon}\left(1+d\varepsilon^\nu
\left(\dfrac{\delta^2}{(2\pi)^d}\right)^{\mu/(2\nu)}\right)^{-1}\int_{B_\varepsilon}(|\xi_1|^\mu+\ldots+
|\xi_d|^\mu)^{2\nu/\mu}\,d\xi\\
\le\frac{\delta^2}{(2\pi)^d}\left(1+d\varepsilon^\nu
\left(\dfrac{\delta^2}{(2\pi)^d}\right)^{\mu/(2\nu)}\right)^{-1}\left(\left(\frac{(2\pi)^d}{\delta^2}\right)^{\mu/(2\nu)}
|\alpha|^{-1}\sum_{j=1}^d\alpha_j+d\varepsilon^\mu\right)^{2\nu/\mu}\\=1.
\end{multline*}
By virtue of \eqref{EDD2a} we have
\begin{multline*}
E_2^2(D^\alpha,\Lambda_\mu^{\nu/2})\ge\|D^\alpha x_\varepsilon\cd\|_{L_2(\mathbb R^d)}^2\\
=\frac{\delta^2}{(2\pi)^d\mes B_\varepsilon}\left(1+d\varepsilon^\nu
\left(\dfrac{\delta^2}{(2\pi)^d}\right)^{\mu/(2\nu)}\right)^{-1}
\int_{B_\varepsilon}|\xi_1|^{2\alpha_1}
\ldots|\xi_d|^{2\alpha_d}\,d\xi\\
=\frac{\delta^2}{(2\pi)^d}\left(1+d\varepsilon^\nu
\left(\dfrac{\delta^2}{(2\pi)^d}\right)^{\mu/(2\nu)}\right)^{-1}|\widetilde\xi_1|^{2\alpha_1}
\ldots|\widetilde\xi_d|^{2\alpha_d},\quad
(\widetilde\xi_1,\ldots,\widetilde\xi_d)\in B_\varepsilon.
\end{multline*}
Letting $\varepsilon\to0$ we obtain the estimate
\begin{equation}\label{EEa}
E_2^2(D^\alpha,\Lambda_\mu^{\nu/2})\ge \left(\frac{\delta^2}{(2\pi)^d}\right)^{1-|\alpha|/\nu}|\alpha|^{-2|\alpha|/\mu}
\prod_{\substack{j=1\\\alpha_j\ne0}}^d\alpha_j^{2\alpha_j/\mu}.
\end{equation}

We will find optimal methods among methods \eqref{mmaa}. Passing to the Fourier transform we have
\begin{equation*}
\|D^\alpha x\cd-\wm(y)\cd\|^2_{\ld}\\
=\frac1{(2\pi)^d}\iRd|\xi_1|^{2\alpha_1}
\ldots|\xi_d|^{2\alpha_d}\left|Fx(\xi)-a(\xi)y(\xi)\right|^2\,d\xi.
\end{equation*}
We set $z\cd=Fx\cd-y\cd$ and note that
\begin{equation*}
\iRd|z(\xi)|^2\,d\xi\le\delta^2,\quad\frac1{(2\pi)^d}\iRd\psi_\mu^\nu(\xi)|Fx(\xi)|^2\,d\xi\le1.
\end{equation*}
Then
\begin{multline*}
\|D^\alpha x\cd-\wm(y)\cd\|^2_{\ld}\\
=\frac1{(2\pi)^d}\iRd|\xi_1|^{2\alpha_1}
\ldots|\xi_d|^{2\alpha_d}\left|\left(1-a(\xi)\right)Fx(\xi)+
a(\xi)z(\xi)\right|^2\,d\xi.
\end{multline*}
We write the integrand as
\begin{equation*}
|\xi_1|^{2\alpha_1}
\ldots|\xi_d|^{2\alpha_d}\left|\frac{(1-a(\xi))\sqrt{\lambda_2}\psi_\mu^{\nu/2}(\xi)
Fx(\xi)}{\sqrt{\lambda_2}\psi_\mu^{\nu/2}(\xi)}+
\frac{a(\xi)}{(2\pi)^{d/2}\sqrt{\lambda_1}}(2\pi)^{d/2}\sqrt{\lambda_1}z(\xi)\right|^2.
\end{equation*}
Applying the Cauchy-Bunyakovskii-Schwarz inequality we obtain the estimate
\begin{multline*}
\|D^\alpha x\cd-\wm(y)\cd\|^2_{\ld}\\
\le\vraisup_{\xi\in\mathbb R^d}S(\xi)\frac1{(2\pi)^d}\iRd\left(\lambda_2\psi_\mu^\nu(\xi)|Fx(\xi)|^2+
(2\pi)^d\lambda_1|z(\xi)|^2\right)\,d\xi,
\end{multline*}
where
\begin{equation*}
S(\xi)=|\xi_1|^{2\alpha_1}
\ldots|\xi_d|^{2\alpha_d}\left(\frac{|1-a(\xi)|^2}{\lambda_2\psi_\mu^\nu(\xi)}+
\frac{|a(\xi)|^2}{(2\pi)^d\lambda_1}\right).
\end{equation*}
If we assume that $S(\xi)\le1$ for almost all $\xi$, then taking into account \eqref{EEa}, we get
\begin{multline}\label{Uba}
e^2_2(D^\alpha ,\Lambda_\mu^{\nu/2},\wm)\\
\le\frac1{(2\pi)^d}\iRd\left(\lambda_2\psi_\mu^\nu(\xi)|Fx(\xi)|^2+
(2\pi)^d\lambda_1|z(\xi)|^2\right)\,d\xi
\le\lambda_2+\lambda_1\delta^2\\=\left(\frac{\delta^2}{(2\pi)^d}\right)^{1-|\alpha|/\nu}|\alpha|^{-2|\alpha|/\mu}
\prod_{\substack{j=1\\\alpha_j\ne0}}^d\alpha_j^{2\alpha_j/\mu}\le E_2^2(D^\alpha,\Lambda_\mu^{\nu/2}).
\end{multline}
This proves \eqref{Ea} and shows that the methods under consideration are optimal.

It remains to verify that the set of functions $a\cd$ satisfying \eqref{aaa} is nonempty. Put
\begin{equation*}
a(\xi)=\frac{(2\pi)^d\lambda_1}{(2\pi)^d\lambda_1+\lambda_2\psi_\mu^\nu(\xi)}.
\end{equation*}
Then
\begin{equation*}
S(\xi)=\frac{|\xi_1|^{2\alpha_1}
\ldots|\xi_d|^{2\alpha_d}}{(2\pi)^d\lambda_1+\lambda_2\psi_\mu^\nu(\xi)}.
\end{equation*}
Consider the function
$$H(t)=-1+(2\pi)^d\lambda_1e^{-2(\alpha,t)}+\lambda_2\biggl(\sum_{j=1}^de^{\mu t_j-\frac\mu\nu(\alpha,t)}\biggr)^{2\nu/\mu},$$
where $(\alpha,t)=\alpha_1t_1+\ldots\alpha_dt_d$. It is easy to prove that $H\cd$ is a convex function. Moreover, $H(\wt)=0$ and the derivative of $H\cd$ at the point $\wt$ is also zero, where
$$\wt=\left(\frac1\mu\ln\frac{\alpha_1}{|\alpha|}+\frac1{2\nu}\ln\frac{(2\pi)^d}{\delta^2},
\ldots,\frac1\mu\ln\frac{\alpha_d}{|\alpha|}+\frac1{2\nu}\ln\frac{(2\pi)^d}{\delta^2}\right).$$
Consequently, $H(t)\ge0$ for all $t\in\mathbb R^d$. It means that
$$e^{-2(\alpha,t)}\le(2\pi)^d\lambda_1+\lambda_2\biggl(\sum_{j=1}^de^{\mu t_j}\biggr)^{2\nu/\mu}.$$
Put $|\xi_j|=e^{t_j}$, $j=1,\ldots,d$. Then we obtain
$$|\xi_1|^{2\alpha_1}
\ldots|\xi_d|^{2\alpha_d}\le(2\pi)^d\lambda_1+\lambda_2\psi_\mu^\nu(\xi).$$
Thus, $S(\xi)\le1$.

The proof of \eqref{Sl2aa} is similar to the proof of \eqref{Sl2}.
\end{proof}

For $\mu=2\nu$ we obtain
\begin{corollary}\label{T3c}
Let $0<|\alpha|<\nu$. Then
$$E_2(D^\alpha,\Lambda_{2\nu}^{\nu/2})= \left(\frac{\delta}{(2\pi)^{d/2}}\right)^{1-|\alpha|/\nu}|\alpha|^{-|\alpha|/(2\nu)}
\prod_{\substack{j=1\\\alpha_j\ne0}}^d\alpha_j^{\alpha_j/(2\nu)},$$
and all methods
$$\wm(y)(t)=F^{-1}\left(a(t)(it)^\alpha y(t)\right),$$
where $a\cd$ are measurable functions satisfying the condition
$$\psi_\theta^\eta(\xi)\left(\frac{|1-a(\xi)|^2}{\lambda_2\psi_{2\nu}^\nu(\xi)}+\frac{|a(\xi)|^2}
{(2\pi)^d\lambda_1}\right)\le1,$$
in which
\begin{align*}
\lambda_1&=\frac1{(2\pi)^d}\left(1-\frac{|\alpha|}\nu\right)
\left(\frac{(2\pi)^d}{|\alpha|\delta^2}
\right)^{|\alpha|/\nu}\prod_{\substack{j=1\\\alpha_j\ne0}}^d\alpha_j^{\alpha_j/\nu},\\
\lambda_2&=\frac{|\alpha|}\nu\left(\frac{(2\pi)^d}{|\alpha|\delta^2}
\right)^{|\alpha|/\nu-1}\prod_{\substack{j=1\\\alpha_j\ne0}}^d\alpha_j^{\alpha_j/\nu},
\end{align*}
are optimal.

The sharp inequality
\begin{multline}\label{Sl2aaa}
\|D^\alpha x\cd\|_{\ld}\\
\le\frac{|\alpha|^{-|\alpha|/(2\nu)}}{(2\pi)^{d(1-|\alpha|/\nu)/2}}
\prod_{\substack{j=1\\\alpha_j\ne0}}^d\alpha_j^{\alpha_j/(2\nu)}
\|Fx\cd\|_{\ld}^{1-|\alpha|/\nu}\biggl(\sum_{j=1}^d\|D^{\nu e_j}x\cd\|_{\ld}^2\biggr)^{|\alpha|/(2\nu)}
\end{multline}
holds.
\end{corollary}

For integer $\nu$ inequality \eqref{Sl2aaa} can be rewritten in the form
\begin{multline*}
\|D^\alpha x\cd\|_{\ld}\\
\le\frac{|\alpha|^{-|\alpha|/(2\nu)}}{(2\pi)^{d(1-|\alpha|/\nu)/2}}
\prod_{\substack{j=1\\\alpha_j\ne0}}^d\alpha_j^{\alpha_j/(2\nu)}
\|Fx\cd\|_{\ld}^{1-|\alpha|/\nu}\biggl(\sum_{j=1}^d\biggl\|\frac{\partial^\nu x}{\partial t_j^\nu}\cd\biggr\|_{\ld}^2\biggr)^{|\alpha|/(2\nu)}.
\end{multline*}

\end{document}